\documentclass[12pt]{amsart}

\usepackage[notref,notcite]{showkeys}

\usepackage{fullpage}
\usepackage{amsfonts}
\usepackage{graphicx}
\usepackage{amsmath, amsthm, amssymb}
\usepackage{fancybox}
\usepackage{verbatim}

\renewcommand{\tilde}{\widetilde}

\newcommand{\C}{\mathbb{C}}
\newcommand{\N}{\mathbb{N}}
\newcommand{\Z}{\mathbb{Z}}

\newcommand{\Hil}{\mathcal{H}}

\DeclareMathOperator{\lspan}{span}
\DeclareMathOperator{\ran}{ran}

\numberwithin{equation}{section}
\theoremstyle{plain} %% This is the default, anyway

\newtheorem{thm}{Theorem}[section]
\newtheorem{cor}[thm]{Corollary}
\newtheorem{lem}[thm]{Lemma}

\theoremstyle{definition}
\newtheorem{defn}{Definition}[section]

\theoremstyle{remark}
\newtheorem{remark}{Remark}[section]
\newtheorem{ex}{Example}

\begin{document}

\title{The Schur-Horn Theorem for Operators with Three Point Spectrum}

\author{John Jasper}

\date{}

\begin{abstract} We characterize the set of diagonals of the unitary orbit of a self-adjoint operator with three points in the spectrum. Our result gives a Schur-Horn theorem for operators with three point spectrum analogous to Kadison's result for orthogonal projections \cite{k1,k2}.
\end{abstract}

\thanks{The author wishes to thank Marcin Bownik for many discussions and comments leading to the improvement of this paper.}

\address{Department of Mathematics, University of Missouri, Columbia, MO 65211--4100, USA}

\email{jasperj@missouri.edu}

\keywords{diagonals of self-adjoint operators, Schur-Horn theorem, Pythagorean theorem, the Carpenter's theorem}

\subjclass[2010]{Primary: 47B15, Secondary: 46C05}
\date{\today}
\date{}

\maketitle

\section{Introduction}

The goal of this paper is to establish an analogue of the Schur-Horn theorem on an infinite dimensional separable Hilbert space for operators with three points in the spectrum. That is, we will give necessary and sufficient conditions for a countable sequence $\{d_{i}\}$ to be the diagonal of a self-adjoint operator with eigenvalues $\{0,A,B\}$ with specified (possibly infinite) multiplicities.

This paper falls into a broader category of research that aims at finding an analogue of the Schur-Horn theorem for operators on a separable infinite dimensional Hilbert space. Recently there has been a great deal of progress by a number of authors. The work of Gohberg and Markus \cite{gm} and Arveson and Kadison \cite{ak} extended the Schur-Horn theorem to positive trace class operators. More recently Kaftal and Weiss \cite{kw} have extended this to all positive compact operators. Antezana, Massey, Ruiz, and Stojanoff \cite{amrs} established a connection between this problem and frame theory, as well as establishing some necessary and sufficient conditions. See \cite{cas,cl,cfklt,kl} for more on this problem from a frame theory perspective. Other work in this area includes the study of $\rm{II}_{1}$ factors by Argerami and Massey \cite{am,am2} and normal operators by Arveson \cite{a}. Neumann \cite{neu} proved what may be considered an approximate Schur-Horn theorem since it is given in terms of the $\ell^{\infty}$-closure of the set of diagonal sequences. Bownik and the author \cite{mbjj} established a variant of the Schur-Horn theorem for the set of locally invertible positive operators.

Of particular interest for our purposes is Kadison's theorem \cite{k1,k2}. For any orthogonal projection $P$, this theorem gives an explicit characterization of the set of diagonal sequences of the unitary orbit of $P$. This can be considered as an infinite dimensional extension of the Schur-Horn theorem for operators with two points in the spectrum. It is a natural next step to consider operators with three points in the spectrum. In this paper we extend the Schur-Horn theorem to such operators.

We would like to emphasize two significant qualitative differences between Kadison's theorem and our extension to operators with three point spectrum. The necessary and sufficient condition for a sequence to be the diagonal of a projection is a single trace condition, that is an equation involving sums of diagonal terms. The requirements for a sequence to be the diagonal of an operator with three point spectrum involve both a trace condition and a majorization inequality.

Also distinct from the case of operators with two point spectrum, it is possible for two non-unitarily equivalent operators with three point spectrum to have the same diagonal. For projections the dimension of the kernel and range (i.e.\ the multiplicities of $0$ and $1$) can be recovered from the diagonal. Indeed, if $\{d_{i}\}$ is the diagonal of a projection $P$, then
\[\dim\ran P = \sum d_{i}\quad\text{and}\quad \dim\ker P = \sum (1-d_{i}).\]
However, for operators with three point spectrum the multiplicities cannot in general be determined from the diagonal, see Remark \ref{rmk2}.

This leads to two distinct extensions of the Schur-Horn theorem for operators with three point spectrum. In the case where the multiplicities of eigenvalues are not given we have the following general theorem characterizing diagonals of operators with three point spectrum.

\begin{thm}\label{3pt} Let $0<A<B<\infty$ and let $\{d_{i}\}_{i\in I}$ be a countable sequence in $[0,B]$ with $\sum d_{i} = \sum (B-d_{i}) = \infty$. Define
\begin{equation}\label{CandD}C = \sum_{d_{i}<A}d_{i}\quad\text{and}\quad D=\sum_{d_{i}\geq A}(B-d_{i}).\end{equation}
There is a positive operator $E$ with diagonal $\{d_{i}\}_{i\in I}$ and $\sigma(E) = \{0,A,B\}$ if and only if one of the following holds:
\begin{enumerate}
\item $C=\infty$
\item $D=\infty$
\item $C,D<\infty$ and there exist $N\in\N$ and $k\in\Z$ such that
\begin{equation}\label{3pttrace} C-D = NA+kB\end{equation}
\begin{equation}\label{3ptmaj} C\geq (N+k)A.\end{equation}
\end{enumerate}
\end{thm}

The assumption that $\sum d_{i} = \sum (B-d_{i}) = \infty$ is not a true limitation. Indeed, the summable case $\sum d_{i}<\infty$ requires more restrictive conditions which can be deduced from parts (a) and (b) of Theorem \ref{fullthm}. For the proof of Theorem \ref{3pt} we refer the reader to Theorem \ref{3ptg}.

Theorem \ref{fullthm} is our second extension of the Schur-Horn theorem which gives a complete list of characterization conditions of diagonals of operators with prescribed multiplicities. Before we state the full theorem, we need one convenient definition.

\begin{defn} Let $E$ be a bounded operator on a Hilbert space. For $\lambda\in\C$ define
\[m_{E}(\lambda) = \dim\ker(E-\lambda).\]
\end{defn}

\begin{thm}\label{fullthm} Suppose $0<A<B<\infty$, let $\{d_{i}\}_{i\in I}$ be a countable (possibly finite) sequence in $[0,B]$, and suppose $N,K,Z\in\N$. Define the sets
\[I_{1} = \{i\in I:d_{i}<A\},\ I_{2} = \{i\in I:d_{i}\geq A\},\ J_{2} = \{i\in I_{2}:d_{i}<(A+B)/2\},\ J_{3} = I_{2}\setminus J_{2}.\]
Let $C$ and $D$ be as in \eqref{CandD} and define the constants (each possibly infinite)
\[C_{1} = \sum_{i\in I_{1}}(A-d_{i}),\ C_{2} = \sum_{i\in J_{2}}(d_{i}-A),\ C_{3} = \sum_{i\in J_{3}}(B-d_{i}).\]
The following table gives necessary and sufficient conditions for $\{d_{i}\}$ to be the diagonal of a positive operator $E$ with $\sigma(E)=\{0,A,B\}$ and the specified multiplicities.
\[
\begin{array}{l|c|c|c|c}
 & m_{E}(0) & m_{E}(A) & m_{E}(B) & \mathrm{Condition}\\
\hline
& & & & \\ [-1.8ex]
(a) & Z & N & K & |I|=Z+N+K\\
& & & & \displaystyle{\sum_{i\in I}d_{i} = NA+KB},\ C\geq (N+K-|I_{2}|)A\\[2.5ex]
\hline
& & & & \\ [-1.8ex]
(b) & \infty & N & K & |I_{1}| = \infty,\\
& & & & \displaystyle{\sum_{i\in I}d_{i} = NA+KB},\ C\geq (N+K-|I_{2}|)A\\ [2.5ex]
\hline
& & & & \\ [-1.8ex]
(c) & \infty & \infty & \infty & C+D=\infty\\ [0.6ex]
\hline
& & & & \\ [-1.8ex]
(d) & \infty & N & \infty & C+D = \infty\\
& & & & or\\
& & & & C,D<\infty,\ |I_{1}|= |I_{2}|=\infty,\\
& & & & \exists\,k\in\Z\text{ such that }C-D=NA+kB,\ C\geq A(N+k)\\
\hline
& & & & \\ [-1.8ex]
(e) & Z & \infty & \infty & C_{1}\leq AZ,\ C_{2}+C_{3} = \infty\\
& & & & or\\
& & & & |I_{1}\cup J_{2}|=|J_{3}|=\infty,\ C_{1}\leq AZ,\ C_{2},C_{3}<\infty\\
& & & & \exists\,k\in\Z\text{ such that }C_{1}-C_{2}+C_{3} = (Z-k)A+kB\\
\hline
& & & & \\ [-1.8ex]
(f) & Z & \infty & K & |I|=\infty,\ C_{1}\leq AZ\\
& & & &  \displaystyle{\sum_{i\in I}(d_{i}-A) = K(B-A)-ZA}\\
\end{array}
\]
\end{thm}

Note that in the preceding theorem we left out the case where only $B$ has infinite multiplicity and the case where only $B$ has finite multiplicity. However, these two remaining cases follow easily using symmetry arguments by applying parts (b) and (e) to the operator $BI-E$ and the sequence $\{B-d_{i}\}$. Also, observe that case (a) corresponds to the finite dimensional case, and hence it is the classical Schur-Horn theorem (for operators with three eigenvalues), albeit written in a new form. Finally, in this paper we only consider the case of separable Hilbert spaces, and thus the indexing set $I$ is always taken to be a countable (possibly finite) set. We will use the notation $C$ and $D$ given in \eqref{CandD} as well as the notation introduced in Theorem \ref{fullthm} throughout the rest of the paper.

The proof of Theorem \ref{fullthm} breaks into 4 distinct parts. The summable cases (a) and (b) do not require many new techniques since they reduce to the study of trace class operators. In Section 3 parts (a) and (b) are relatively easily deduced from the work of Arveson-Kadison \cite{ak}. The remaining 3 parts rely heavily on a technique, which was introduced in \cite{mbjj}, of ``moving'' diagonal entries to more favorable configurations (see Lemma \ref{ops}), where it is possible to construct required operators. In Section 4 we deal with the case (c) involving three (or more) eigenvalues of infinite multiplicity. The necessity of the condition in (c) follows from Theorem \ref{nec4}, the sufficiency follows from Theorem \ref{suff2}. Much more involved combinatorial arguments are needed in Section 5 to deal with case (d) involving two outer eigenvalues with infinite multiplicities. Part (d) is proved in Theorem \ref{N<infty}. In Section 6 we analyze the cases (e) and (f) where at least one of outer eigenvalues has finite multiplicity. The proofs of the necessity and the sufficiency in these last two cases require even more subtle combinatorial arguments which is partially evidenced by the complicated nature of the characterization conditions. In Theorem \ref{nec7} we show that the conditions in part (e) are necessary, while in Theorem \ref{suff4} we show that they are sufficient. Finally, part (f) of Theorem \ref{fullthm} is proved in Corollary \ref{part(f)}.

We finish the paper with an application of Theorem \ref{fullthm} in Section 7. Given a sequence $\{d_{i}\}$ in $[0,1]$ we are interested in determining the set of numbers $A\in(0,1)$ for which there exists a positive operator with spectrum $\{0,A,1\}$ and diagonal $\{d_i\}$. We show that this set is either finite or the full open interval $(0,1)$. Finally, we look at some specific sequences $\{d_{i}\}$ and explicitly calculate the set of possible $A$.

\section{Preliminaries}

Our arguments rely on the classical Schur-Horn theorem \cite{horn,schur}, which we state here.

\begin{thm}[Schur, Horn]\label{sh} Let $N\in\N$ and let $\{\lambda_{i}\}_{i=1}^{N}$ and $\{d_{i}\}_{i=1}^{N}$ be nonincreasing sequences of real numbers. There is an $N\times N$ self-adjoint matrix with eigenvalues $\{\lambda_{i}\}$ and diagonal $\{d_{i}\}$ if and only if
\begin{equation}\label{fmaj}\begin{split}
\sum_{i=1}^{n}d_{i} \leq \sum_{i=1}^{n}\lambda_{i} & \quad \text{for all }n\leq N\\
\sum_{i=1}^{N}d_{i} = \sum_{i=1}^{N}\lambda_{i}. & \\
\end{split}\end{equation}
\end{thm}

In fact, we need a version of the Schur-Horn theorem for finite rank operators. This can be deduced from a theorem of Arveson and Kadison \cite[Theorem 4.1]{ak} or from a theorem of Kaftal and Weiss \cite[Theorem 6.1]{kw}.

\begin{thm}[Arveson, Kadison]\label{frsh} Let $\{\lambda_{i}\}_{i=1}^{N}$ be positive and nonincreasing. Let $\{d_{i}\}_{i=1}^{\infty}$ be nonnegative and nonincreasing. There is a rank $N$ positive operator with positive eigenvalues $\{\lambda_{i}\}$ and diagonal $\{d_{i}\}$ if and only if
\begin{equation}\label{frmaj}\begin{split}
\sum_{i=1}^{n}d_{i} \leq \sum_{i=1}^{n}\lambda_{i} & \quad \text{for all }n\leq N\\
\sum_{i=1}^{\infty}d_{i} = \sum_{i=1}^{N}\lambda_{i}. & \\
\end{split}\end{equation}
\end{thm}

We will also make extensive use of Kadison's theorem \cite{k1,k2}.

\begin{thm}[Kadison]\label{Kadison} Let $\{d_{i}\}_{i\in I}$ be a sequence in $[0,1]$ and $\alpha\in(0,1)$. Define
\[a=\sum_{d_{i}<\alpha}d_{i} \qquad b=\sum_{d_{i}\geq \alpha}(1-d_{i}).\]
There is a projection with diagonal $\{d_{i}\}_{i\in I}$ if and only if \begin{equation}\label{kadcond} a-b\in\Z\cup\{\pm\infty\},\end{equation}
with the convention that $\infty-\infty=0$.
\end{thm}

\begin{remark} Observe that in Theorem \ref{Kadison}, if there exists a partition of $I=I_{1}\cup I_{2}$ such that
\[\sum_{i\in I_{1}}d_{i},\sum_{i\in I_{2}}(1-d_{i})<\infty,\quad\text{and}\quad\sum_{i\in I_{1}} d_{i} - \sum_{i\in I_{2}}(1-d_{i})\in\Z,\]
then we have $a-b\in\Z$ for all $\alpha\in(0,1)$. Thus, the existence of such a partition is also a sufficient condition for a sequence to be the diagonal of a projection. We will find use for these more general partitions in the sequel.\end{remark}

\section{Finite Rank Operators}

The following is an application of Theorems \ref{sh} and \ref{frsh}, which establishes parts (a) and (b) of Theorem \ref{fullthm}. Theorem \ref{fr3pt} is the analogue of Theorem 3.4 from \cite{mbjj} which characterizes the diagonals of finite rank operators such that $\{A,B\}\subset \sigma(E)\subset\{0\}\cup[A,B]$. 

\begin{thm}\label{fr3pt} Let $0<A<B<\infty$, let $\{d_{i}\}_{i\in I}$ be a summable sequence in $[0,B]$, and let $N,K\in\N$ with $N+K<|I|$. There is a positive rank $N+K$ operator $E$ with diagonal $\{d_{i}\}$, $\sigma(E)=\{0,A,B\}$, $m_{E}(A)=N$, and $m_{E}(B)=K$ if and only if
\begin{equation}\label{fr3pt1}\sum_{i\in I}d_{i}=NA+KB\end{equation}
\begin{equation}\label{fr3pt2}\sum_{d_{i}<A}d_{i}\geq (N+K-n_{0})A,\end{equation}
where $n_{0}=|\{i:d_{i}\geq A\}|$.
\end{thm}

\begin{proof} We will first prove the theorem under the assumption that $\{d_{i}\}$ can be arranged in nonincreasing order. Setting $M = |I|\in\N\cup\{\infty\}$, we may assume our sequence is given by $\{d_{i}\}_{i=1}^{M}$ in nonincreasing order.

To prove that \eqref{fr3pt1} and \eqref{fr3pt2} are necessary, assume $E$ is a positive operator with diagonal $\{d_{i}\}_{i\in I}$, $\sigma(E)=\{0,A,B\}$, $m_{E}(A)=N$ and $m_{E}(B)=K$. The operator $E$ has finite rank, hence it is of trace class with trace equal to $NA+KB$; this is \eqref{fr3pt1}. The eigenvalues sequence of $E$ written in nonincreasing order is given by
\begin{equation}\label{fr3pt4}\lambda_{i}=\begin{cases}
B & i=1,2,\ldots,K\\
A & i=K+1,\ldots,K+N\\
0 & i>K+N.
\end{cases}\end{equation}
Using Theorem \ref{sh} (or Theorem \ref{frsh} if $|I|=\infty$) we see that
\begin{equation}\label{fr3pt3}\sum_{d_{i}\geq A}d_{i} = \sum_{i=1}^{n_{0}}d_{i} \leq \sum_{i=1}^{n_{0}}\lambda_{i}\leq KB+(n_{0}-K)A.\end{equation}
To see the last inequality in \eqref{fr3pt3}, consider separately the cases where $n_{0}\geq K$ and $n_{0}<K$. Using \eqref{fr3pt3} we have
\[\sum_{d_{i}<A}d_{i} = NA+KB-\sum_{d_{i}\geq A}d_{i}\geq NA+KB-(KB+(n_{0}-K)A)=A(N+K-n_{0}),\]
which is \eqref{fr3pt2}.

Next, assume that \eqref{fr3pt1} and \eqref{fr3pt2} hold. Define the sequence $\{\lambda_{i}\}_{i=1}^{M}$ as in \eqref{fr3pt4}. By Theorem \ref{sh} or Theorem \ref{frsh} it is enough to show that
\begin{equation}\label{eq1}\sum_{i=1}^{m}d_{i}\leq\sum_{i=1}^{m}\lambda_{i}\end{equation}
for all $m\leq M$, since the second condition in either \eqref{fmaj} or \eqref{frmaj} follows from the assumption \eqref{fr3pt1}. Note that \eqref{eq1} holds for $m\leq K$, since $d_{i}\leq B$ for all $i\in I$. For $m>K+N$ we have
\[\sum_{i=1}^{m}d_{i}\leq \sum_{i=1}^{M}d_{i}=\sum_{i=1}^{M}\lambda_{i}=\sum_{i=1}^{m}\lambda_{i},\]
so \eqref{eq1} holds for $m>K+N$.

First, we wish to show that \eqref{eq1} holds for $m=n_{0}$. From the above we may assume $K<n_{0}\leq K+N$. Using \eqref{fr3pt2} we have
\begin{equation*}
\sum_{i=1}^{n_{0}}d_{i}  = NA+KB-\sum_{d_{i}<A}d_{i} \leq NA+KB-A(N+K-n_{0}) = KB + (n_{0}-K)A = \sum_{i=1}^{n_{0}}\lambda_{i}.
\end{equation*}
Now, if $K<m<n_{0}$ then we have
\begin{equation*}
\sum_{i=1}^{m}d_{i} = \sum_{i=1}^{n_{0}}d_{i}-\sum_{i=m+1}^{n_{0}}d_{i}\leq \sum_{i=1}^{n_{0}}\lambda_{i}-(n_{0}-m)A=\sum_{i=1}^{m}\lambda_{i}.
\end{equation*}
Finally, if $n_{0}<m\leq K+N$ then
\begin{equation*}
\sum_{i=1}^{m}d_{i} = \sum_{i=1}^{n_{0}}d_{i}+\sum_{i=n_{0}+1}^{m}d_{i}\leq \sum_{i=1}^{n_{0}}\lambda_{i}+(m-n_{0})A=\sum_{i=1}^{m}\lambda_{i}.
\end{equation*}

To complete the proof we assume $\{d_{i}\}$ cannot be arranged in nonincreasing order. This is the case exactly when $\{d_{i}\}$ has infinitely many nonzero terms and some terms equal to zero.

Assume we have an operator $E$ with diagonal $\{d_{i}\}$, $\sigma(E)=\{0,A,B\}$, $m_{E}(A)=N$ and $m_{E}(B)=K$. Let $\{e_{i}\}_{i\in I}$ be an orthonormal basis such that $d_{i} =\langle Ee_{i},e_{i}\rangle$ for each $i\in I$. Set $I_{0} = \{i\in I:d_{i} = 0\}$. Since $E$ is positive, $e_{i}\in\ker E$ for each $i\in I_{0}$, and thus $\overline{\lspan}\{e_{i}\}_{i\in I\setminus I_{0}}$ is invariant under $E$. Let $E'$ be $E$ acting on the space $\overline{\lspan}\{e_{i}\}_{i\in I\setminus I_{0}}$. The operators $E$ and $E'$ have the same multiplicities at $A$ and $B$, and $E'$ has diagonal $\{d_{i}\}_{i\in I\setminus I_{0}}$. The diagonal of $E'$ is a strictly positive summable sequence, and thus it can be arranged in nonincreasing order. By the above argument, we see that \eqref{fr3pt1} and \eqref{fr3pt2} hold for $\{d_{i}\}_{i\in I\setminus I_{0}}$. Clearly this implies that they hold for the full sequence $\{d_{i}\}_{i\in I}$.

Finally, assume that \eqref{fr3pt1} and \eqref{fr3pt2} hold. The sequence $\{d_{i}\}_{i\in I\setminus I_{0}}$ also satisfies \eqref{fr3pt1} and \eqref{fr3pt2}. Moreover, $\{d_{i}\}_{i\in I\setminus I_{0}}$ can be arranged in nonincreasing order. By the above argument, there is a positive operator $E'$ with diagonal $\{d_{i}\}_{i\in I\setminus I_{0}}$, $\sigma(E') \subset \{0,A,B\}$, $m_{E'}(A)=N$, and $m_{E'}(B)=K$. Let $\bf{0}$ be the zero operator on a separable Hilbert space with dimension $|I_{0}|$. The operator $E=E'\oplus\bf{0}$ has the desired spectral properties and diagonal. \end{proof}

\section{Three or more eigenvalues of infinite multiplicity}

In this section we will classify the diagonals of operators with exactly three eigenvalues, each with infinite multiplicity. This will yield part (c) of Theorem \ref{fullthm}. We will also show that a sequence with $C+D=\infty$ is the diagonal of a very general class of operators.

Theorem \ref{nec4} is the analogue of \cite[Theorem 5.1]{mbjj}, and it is used in showing the necessity of part (c) of Theorem \ref{fullthm}. In particular, Theorem \ref{nec4} shows that $C,D<\infty$ implies that only $0$ and $B$ can have infinite multiplicity. Thus, $C+D=\infty$ is a necessary condition for a sequence to be the diagonal of a self-adjoint operator with at least three infinite multiplicities.

\begin{thm}\label{nec4} Let $0<A<B<\infty$ and let $E$ be a positive operator on a Hilbert space $\Hil$ with $\sigma(E)=\{0,A,B\}$. Let $\{e_{i}\}_{i\in I}$ be an orthonormal basis for $\Hil$ and set $d_{i}=\langle Ee_{i},e_{i}\rangle$. If $C,D<\infty$ then $N:=m_{E}(A)<\infty$ and there is some $k\in\Z$ such that
\begin{equation}\label{nec4.1}C-D = NA + kB ,\end{equation}
\begin{equation}\label{nec4.5}C\geq (N+k)A.\end{equation}
\end{thm}

\begin{proof} Define the sets $I_{1}=\{i:d_{i}<A\}$ and $I_{2} = \{i:d_{i}\geq A\}$. Let $P$ be the orthogonal projection onto $\ker(E-A)$, and let $Q$ be the projection onto $\ker(E-B)$. This yields the decomposition $E=AP+BQ$. Define $p_{i} = \langle Pe_{i},e_{i}\rangle$ and $q_{i} = \langle Qe_{i},e_{i}\rangle$, so that $d_{i} = Ap_{i}+Bq_{i}$. By \cite[Theorem 5.1]{mbjj}, the operator $B(Q+P) - E = (B-A)P$ is of trace class and thus finite rank. From this we conclude

\begin{equation}\label{nec4.2}N = m_{E}(A) = \sum_{i\in I}p_{i}<\infty.\end{equation}

Define
\[a:=\sum_{i\in I_{1}}q_{i} = \frac{1}{B}\sum_{i\in I_{1}}(d_{i} - Ap_{i}) \leq \frac{1}{B}\sum_{i\in I_{1}}d_{i} = \frac{C}{B}<\infty,\]
and
\[b: = \sum_{i\in I_{2}}(1-q_{i}) = \frac{1}{B}\sum_{i\in I_{2}}(B-d_{i} + Ap_{i})\leq \frac{D}{B}+ \frac{A}{B}\sum_{i\in I_{2}}p_{i}.\]
Using \eqref{nec4.2} we see that $b<\infty$. By Theorem \ref{Kadison} there exists $k\in\Z$ such that $a-b = k$.

Now, we calculate
\begin{align*}
C-D & = \sum_{i\in I_{1}}(Ap_{i} + Bq_{i}) - \sum_{i\in I_{2}}\left(B-Ap_{i} - Bq_{i}\right)\\
 & = \sum_{i\in I}Ap_{i} + B\left(\sum_{i\in I_{1}}q_{i} - \sum_{i\in I_{2}}(1-q_{i})\right) = NA + kB,\\
\end{align*}
which shows \eqref{nec4.1}.

Finally, we calculate
\begin{align*}k(B-A) + D & = (a-b)(B-A) + \sum_{i\in I_{2}}(B-Bq_{i} - Ap_{i}) = a(B-A)+bA-\sum_{i\in I_{2}}Ap_{i}\\
 & \geq  bA - bB + bB - \sum_{i\in I_{2}}Ap_{i} = bA - \sum_{i\in I_{2}}Ap_{i} = A\sum_{i\in I_{2}}(1-p_{i}-q_{i}).
\end{align*}
Together with the fact that $p_{i} + q_{i}\leq 1$, this shows $k(B-A)+D\geq 0$, or $kB+D\geq kA$. Combining this with \eqref{nec4.1} gives \eqref{nec4.5}.
\end{proof}

Next, we will show that the condition $C+D=\infty$ is sufficient for $\{d_{i}\}$ to be the diagonal of any diagonalizable self-adjoint operator with the property that the largest and smallest eigenvalues have infinite multiplicity. In particular, we will prove the following theorem, which will complete the proof of part (c) of Theorem \ref{fullthm}.

\begin{thm}\label{suff2} Let $\Lambda\subset[0,B]$ be a countable set with $0,B\in \Lambda$. Set $n_{0}=n_{B}=\infty$, and for each $\lambda\in \Lambda\cap(0,B)$ let $n_{\lambda}\in\N\cup\{\infty\}$. If $\{d_{i}\}_{i\in I}$ is a sequence in $[0,B]$ such that for some (and hence all) $\alpha\in (0,B)$ we have
\[\sum_{d_{i}<\alpha}d_{i} + \sum_{d_{i}\geq\alpha}(B-d_{i})=\infty,\]
then there is a positive diagonalizable operator $E$ with diagonal $\{d_{i}\}$, eigenvalues $\Lambda$ and $m_{E}(\lambda)=n_{\lambda}$ for each $\lambda\in \Lambda$.
\end{thm}

To prove Theorem \ref{suff2} we need two lemmas. The first lemma will also be used in later sections, for the proof see \cite[Lemmas 4.3 and 4.4]{mbjj}.

\begin{lem}\label{ops}
 Let $\{d_{i}\}_{i\in I}$ be a sequence in $[0,B]$. Let $F_0, F_1 \subset I$ be two disjoint finite subsets such that $\max\{d_{i}: i \in I_0\}\leq\min\{d_{i}: i \in I_1\}$. Let $\eta_{0}\geq 0$ and
\[
\eta_{0}\leq\min\bigg\{\sum_{i\in F_0} d_{i},\sum_{i\in F_1}
(B-d_{i}) \bigg\}.\]
(i) There exists a sequence $\{\tilde d_{i}\}_{i\in I}$  in $[0,B]$
satisfying
\begin{align}
\label{ops0}
\tilde d_i = d_i &\quad\text{for } i \in I \setminus (F_0 \cup F_1),
\\
\label{ops1}
\tilde d_i \leq d_i \quad i\in F_0,
&\quad\text{and}\quad
\tilde d_i \ge d_i, \quad i\in F_1,
\\
\label{ops2}
\eta_0+\sum_{i\in F_0}\tilde{d}_{i} =\sum_{i\in F_0}d_{i}
&\quad\text{and}\quad 
\eta_0+\sum_{i\in F_1} (B-\tilde{d}_{i})=\sum_{i\in F_1} (B-d_{i}).
\end{align}
(ii) For any self-adjoint operator $\tilde E$ on $\Hil$ with diagonal $\{\tilde d_{i}\}_{i\in I}$,
there exists an operator $E$ on $\Hil$ unitarily equivalent to $\tilde E$ with diagonal $\{d_{i}\}_{i\in I}$.
\end{lem}
The second lemma will serve as a building block for constructing the operators in Theorem \ref{suff2}.

\begin{lem}\label{suff1} Let $0<A<B<\infty$ and let $\{d_{i}\}_{i\in I}$ be a sequence in $[0,B]$. If $C+D=\infty$ then there is a positive operator $E$ with $\sigma(E)=\{0,A,B\}$, $m_{E}(0) = m_{E}(B) = \infty$, $m_{E}(A) = 1$, and diagonal $\{d_{i}\}$.\end{lem}

\begin{proof} Assume $C=\infty$. There exists $i_{0}\in I_{1}=\{i\in I:d_{i}<A\}$ such that
\[\sum_{d_{i}\leq d_{i_{0}}}d_{i} >A.\]
This implies that
\[\sum_{\genfrac{}{}{0pt}{}{d_{i}\leq d_{i_{0}}}{i\neq i_{0}}}d_{i}> A-d_{i_{0}}.\]
Let $F_{0}$ be a finite subset of $\{i\in I_{1}\setminus\{i_{0}\}:d_{i}\leq d_{i_{0}}\}$ such that
\[\sum_{i\in F_{0}}d_{i}>A-d_{i_{0}}.\]
Apply Lemma \ref{ops} (i) with $F_{0}$ as above, $F_{1} = \{i_{0}\}$, and $\eta_{0}=A-d_{i_{0}}$ to obtain a sequence $\{\tilde{d}_{i}\}_{i\in I}$. Note that $\tilde{d}_{i_{0}}=A$ and since $F_{0}$ is finite
\[\sum_{i\in I_{1}\setminus\{i_{0}\}}\tilde{d}_{i} = \infty.\]
Theorem \ref{Kadison} implies there is a projection $Q$ such that $BQ$ has diagonal $\{\tilde{d}_{i}\}_{i\in I\setminus\{i_{0}\}}$. Moreover, we have
\[\dim\ker P = \sum_{i\in I\setminus\{i_{0}\}}\left(1-\frac{1}{B}\tilde{d}_{i}\right) = \infty = \frac{1}{B}\sum_{i\in I\setminus\{i_{0}\}}\tilde{d}_{i} = \dim\ran P.\]
Let $P$ be the identity on a one-dimensional Hilbert space. The operator $\tilde{E} = BQ\oplus AP$ has diagonal $\{\tilde{d}_{i}\}_{i\in I}$ as well as the desired spectrum and multiplicities. Finally, by Lemma \ref{ops} (ii) there is an operator $E$, unitarily equivalent to $\tilde{E}$, with diagonal $\{d_{i}\}_{i\in I}$. This completes the proof of the theorem when $C=\infty$.

Assume $D=\infty$. Define $d_{i}' = B-d_{i}$ for each $i\in I$. We have
\[\sum_{d_{i}'\leq B-A}d_{i}'=\sum_{d_{i}\geq A}(B-d_{i}) = D = \infty.\]
By the previous argument, there is a positive operator $E'$ with diagonal $\{d_{i}'\}$ and $\sigma(E')=\{0,B-A,B\}$, with $0$ and $B$ having infinite multiplicity and $B-A$ having multiplicity $1$. Clearly $E=B-E'$ has the desired properties.
\end{proof}

\begin{proof}[Proof of Theorem \ref{suff2}] If $\Lambda=\{0,B\}$ then Theorem \ref{Kadison} gives the desired operator. Thus, we may assume there is some $\lambda\in\Lambda$ with $0<\lambda<B$. Set $I_{1}=\{i\in I:d_{i}<\alpha\}$ and $I_{2}=\{i:d_{i}\geq\alpha\}$. Partition $I_{1}$ and $I_{2}$ into (possibly empty) sets $\{I_{1}^{\lambda}\}_{\lambda\in \Lambda}$ and $\{I_{2}^{\lambda}\}_{\lambda\in \Lambda}$ respectively, such that for each $\lambda\in \Lambda$
\[\sum_{i\in I_{1}^{\lambda}}d_{i} + \sum_{i\in I_{2}^{\lambda}}(B-d_{i})=\infty.\]
For each $\lambda\in \Lambda\cap(0,B)$ partition $I_{1}^{\lambda}$ and $I_{2}^{\lambda}$ into $n_{\lambda}$ (possibly empty) sets $\{I_{1}^{\lambda,n}\}_{n=1}^{n_{\lambda}}$ and $\{I_{2}^{\lambda,n}\}_{n=1}^{n_{\lambda}}$ respectively, such that for each $n=1,2,\ldots,n_{\lambda}$ we have
\[\sum_{i\in I_{1}^{\lambda,n}}d_{i} + \sum_{i\in I_{2}^{\lambda,n}}(B-d_{i})=\infty.\]
By Lemma \ref{suff1}, for each $\lambda\in \Lambda\cap(0,B)$ and each $n=1,2,\ldots,n_{\lambda}$ there is a self-adjoint operator $E_{\lambda,n}$ with diagonal $\{d_{i}\}_{i\in I_{1}^{\lambda,n}\cup I_{2}^{\lambda,n}}$ and $\sigma(E_{\lambda,n})=\{0,\lambda,B\}$ with infinite multiplicity at $0$ and $B$ and multiplicity $1$ at $\lambda$. Finally, set
\[E=\bigoplus_{\lambda\in \Lambda}\bigoplus_{n=1}^{n_{\lambda}}E_{\lambda,n},\]
and it is clear that $E$ has the desired diagonal and eigenvalues.
\end{proof}

In Theorem \ref{suff2} the spectrum of $E$ is the closure of $\Lambda$. To end this section we note that $C+D=\infty$ is a sufficient condition on a sequence to be the diagonal of a positive operator $E$ with $\sigma(E)=K$ for any compact set $K\subset[0,B]$. Simply let $\Lambda$ be a countable dense subset of $K$ and apply Theorem \ref{suff2} with any multiplicities $\{n_{\lambda}\}_{\lambda\in \Lambda}$. This gives us the following corollary.

\begin{cor}\label{suff3} Let $K\subset[0,B]$ be a compact set with $0,B\in K$. If $\{d_{i}\}_{i\in I}$ is a sequence in $[0,B]$ such that for some (and hence all) $\alpha\in (0,B)$ we have
\[\sum_{d_{i}<\alpha}d_{i} + \sum_{d_{i}\geq\alpha}(B-d_{i})=\infty,\]
then there is a positive diagonalizable operator $E$ with diagonal $\{d_{i}\}$ and $\sigma(E)=K$.
\end{cor}

\section{Outer eigenvalues with infinite multiplicity}

In this section we will establish part (d) of Theorem \ref{fullthm}, which is formulated in Theorem \ref{N<infty} below. Moreover, the proof of Theorem \ref{3pt} is given.

\begin{thm}\label{N<infty} Let $0<A<B<\infty$, let $\{d_{i}\}_{i\in I}$ be a sequence in $[0,B]$ and let $N\in\N$. There is a positive operator $E$ with $\sigma(E)=\{0,A,B\}$, $m_{E}(0)=m_{E}(B)=\infty$, $m_{E}(A)=N$, and diagonal $\{d_{i}\}_{i\in I}$ if and only if one of the following holds: 
\begin{enumerate}
\item $C+D=\infty$
\item $C,D<\infty$, $\sum d_{i}=\sum(B-d_{i})=\infty$, and there exists $k\in\Z$ such that
\begin{align}\label{cdfin2}C-D & = NA + kB\\
\label{cdfin3}C & \geq A(N+k).\end{align}\end{enumerate}
\end{thm}

\begin{proof} First, we note that the necessity direction is immediate. Indeed, if (i) fails then we have $C,D<\infty$ and we use Theorem \ref{nec4} to deduce \eqref{cdfin2} and \eqref{cdfin3}. Moreover, $\{d_{i}\}$ and $\{B-d_{i}\}$ are not summable since both $E$ and $B-E$ are positive operators with infinite dimensional range and finite spectrum, and thus they both have infinite trace.

Next, note that Theorem \ref{suff2} implies that (i) is sufficient. Lastly, we assume that (ii) holds, and we must show that the desired operator exists. However, the proof is quite complicated and requires considering four distinct cases. First, we make a couple of observations.

Recall that $I_{1} = \{i:d_{i}<A\}$ and $I_{2} = \{i:d_{i}\geq A\}$. Since $C,D<\infty$ and $\sum d_{i} = \sum (B-d_{i})=\infty$ it must be the case that $|I_{1}|=|I_{2}|=\infty$.

The following argument shows that it is enough to consider sequences $\{d_{i}\}$ with limit points at both $0$ and $B$. Assume $B$ is not a limit point of $\{d_{i}\}$. Since $D<\infty$, the set $I_{2}^{0}:=\{i\in I_{2}: d_{i}<B\}$ is finite. Assume $I_{2}^{0}$ has $M$ elements. Let $L\subset I_{2}\setminus I_{2}^{0}$ be a set with $|k|+1$ elements and define $K_{2}: = I_{2}^{0}\cup L$. If we consider the sequence $\{d_{i}\}_{i\in I_{1}\cup K_{2}}$, then we have
\begin{align*}
\sum_{i\in I_{1}\cup K_{2}}d_{i} & = C + (M+|k|+1)B - \sum_{i\in K_{2}}(B-d_{i})\\
 & = C + (M+|k|+1)B -D = NA + (M+|k|+k+1)B
\end{align*}
and
\[\sum_{\genfrac{}{}{0pt}{}{i\in I_{1}\cup K_{2}}{d_{i}<A}}d_{i} = C \geq (N+k)A = (N+M+|k|+k+1 - |K_{2}|)A.\]
By Theorem \ref{fr3pt}, there is a positive operator $E'$ with diagonal $\{d_{i}\}_{i\in I_{1}\cup K_{2}}$, $\sigma(E') = \{0,A,B\}$, $m_{E'}(0) = \infty$, $m_{E'}(A) = N$ and $m_{E'}(B) = M+|k|+k+1$. Let $I$ be the identity operator on an infinite dimensional Hilbert space. The operator $E = E'\oplus BI$ is as desired.

If $0$ is not a limit point, then we can use the above argument on the sequence $\{B-d_{i}\}$ to obtain an operator $F$ with diagonal $\{B-d_{i}\}$ and eigenvalues $0,B-A$ and $B$ which have multiplicities $\infty,N$ and $\infty$, respectively. Then $B-F$ is the desired operator. For the rest of the proof we can and will assume that both $0$ and $B$ are limit points of $\{d_{i}\}$.

{\bf{Case 1:}} Assume $k\geq0$. Since $B$ is a limit point of $\{d_{i}\}$ we have $D>0$ and thus $C=NA+kB+D>NA+kB$. There is a finite set $F_{0}\subset I_{1}$ such that
\[\sum_{i\in F_{0}}d_{i}> NA+kB\]
Since $0$ is a limit point of $\{d_{i}\}_{i\in I_{1}}$ and $F_{0}$ is finite we have $\sum_{i\in F_{0}}d_{i}<C$. Define
\[\eta_{0} := \sum_{i\in F_{0}}d_{i}-NA-kB < C-NA-kB = D.\]
There is a finite set $F_{1}\subset I_{2}$ such that
\[\sum_{i\in F_{1}}(B-d_{i})>\eta_{0}.\]
The sequences $\{d_{i}\}_{i\in F_{0}}$ and $\{d_{i}\}_{i\in F_{1}}$ are in $[0,B]$, satisfy $\max\{d_{i}\}_{i\in F_{0}}\leq\min\{d_{i}\}_{i\in F_{1}}$ and
\[\eta_{0}\leq\max\left\{\sum_{i\in F_{0}}d_{i},\sum_{i\in F_{1}}(B-d_{i})\right\}.\]
Apply Lemma \ref{ops} (i) with $F_{0}$, $F_{1}$, and $\eta_{0}$ as above, to obtain a sequence $\{\tilde{d}_{i}\}_{i\in I}$. From \eqref{ops2} we have
\[\sum_{i\in F_{0}}\tilde{d}_{i} = \left(\sum_{i\in F_{0}}d_{i}\right)-\eta_{0} = NA+kB.\]
We wish to apply Theorem \ref{fr3pt} to the sequence $\{\tilde{d}_{i}\}_{i\in F_{0}}$, and this shows that \eqref{fr3pt1} holds. From \eqref{ops1} we see that $\tilde{d}_{i}\leq d_{i}<A$ for all $i\in F_{0}$. From this it is it is clear that \eqref{fr3pt2} also holds. We conclude that there is a positive operator $\tilde{E}_{0}$ with diagonal $\{\tilde{d}_{i}\}_{i\in F_{0}}$, $\sigma(\tilde{E}_{0}) = \{0,A,B\}$, $m_{\tilde{E}_{0}}(B)=k$, $m_{\tilde{E}_{0}}(A)=N$ and $m_{\tilde{E}_{0}}(0) = |F_{0}|-k-N$. Using \eqref{ops0} we have 
\[\sum_{i\in I_{1}\setminus F_{0}}\tilde{d}_{i} = \sum_{i\in I_{1}\setminus F_{0}}d_{i} = C-\sum_{i\in F_{0}}d_{i} = D -\eta_{0}\]
and from \eqref{ops2} we see that
\[\sum_{i\in I_{2}}(B-\tilde{d}_{i}) = D-\eta_{0}.\]
By Theorem \ref{Kadison} there is a projection $Q$ such that $BQ$ has diagonal $\{\tilde{d}_{i}\}_{i\in (I_{1}\setminus F_{0}) \cup I_{2}}$. Since $|I_{1}\setminus F_{0}|=|I_{2}|=\infty$ we have $m_{Q}(1)=m_{Q}(0)=\infty$. Thus, the operator $\tilde{E} = E_{0} \oplus BQ$ has the desired eigenvalues and multiplicities and diagonal $\{\tilde{d}_{i}\}_{i\in I}$. Finally, use the second part of Lemma \ref{ops} to obtain an operator $E$, unitarily equivalent to $\tilde{E}$, with diagonal $\{d_{i}\}_{i\in I}$. This completes the proof of the first case.

{\bf{Case 2:}} Assume $k\leq-N$. We obtain this case by applying Case $1$ to the sequence $\{B-d_{i}\}$, to obtain the operator $E_{0}$ with $\sigma(E_{0})=\{0,B-A,B\}$, $\dim\ker(E_{0})=\dim\ker(B-E_{0})=\infty$ and $\dim\ker((B-A)-E_{0})=N$. The operator $B-E_{0}$ has the desired diagonal, eigenvalues, and multiplicities.

{\bf{Case 3:}} Assume $-N<k<0$ and $C=A(N+k)$. Theorem \ref{Kadison} implies there is a projection $P$ with $N+k$ dimensional range, such that $AP$ has diagonal $\{d_{i}\}_{i\in I_{1}}$. Since $|I_{1}|=\infty$ we also see that $P$ has infinite dimensional kernel.

Next, note that
\[\sum_{i\in I_{2}}(B-d_{i}) = D = C - NA-kB = NA+kA-NA-kB = -k(B-A).\]
Theorem \ref{Kadison} implies that there is a projection $Q$ with $-k$ dimensional range, such that $(B-A)Q$ has diagonal $\{B-d_{i}\}_{i\in I_{2}}$. Since $|I_{2}|=\infty$ we see that $Q$ has infinite dimensional kernel. The operator $E = AP\oplus(BI-(B-A)Q)$ has the desired diagonal, eigenvalues, and multiplicities.

{\bf{Case 4:}} Assume $-N<k<0$ and $C>A(N+k)$. Set $\eta_{0}: = C-(N+k)A<C$.
There is a finite set $F_{0}\subset I_{1}$ such that
\[\sum_{i\in F_{0}}d_{i}>\eta_{0}.\]
Next, note that
\[\eta_{0} = C-(N+k)A = NA+kB+D-NA-kA = D + k(B-A)<D.\]
Thus, there is a finite set $F_{1}\subset I_{2}$ such that
\[\sum_{i\in F_{1}}(B-d_{i})>\eta_{0}.\]
Apply Lemma \ref{ops} (i) with $F_{0}$, $F_{1}$, and $\eta_{0}$ as above, to obtain a sequence $\{\tilde{d}_{i}\}_{i\in I}$. Using \eqref{ops2} we have
\[\sum_{i\in I_{1}}\tilde{d}_{i}= \sum_{i\in I_{1}\setminus F_{0}}d_{i} + \sum_{i\in F_{0}}\tilde{d}_{i} = \sum_{i\in I_{1}\setminus F_{0}}d_{i} + \sum_{i\in F_{0}}d_{i}-\eta_{0} = C-\eta_{0} = (N+k)A\]
and
\[\sum_{i\in I_{2}}(B-\tilde{d}_{i}) = \sum_{i\in I_{2}\setminus F_{1}}(B-d_{i}) + \sum_{i\in F_{1}}(B-d_{i}) - \eta_{0} = D-\eta_{0} = -k(B-A).\]
Thus, the sequence $\{\tilde{d}_{i}\}_{i\in I}$ satisfies the conditions of Case 3, so there is an operator $\tilde{E}$ with the desired eigenvalues and multiplicities but with diagonal $\{\tilde{d}_{i}\}_{i\in I}$. The second part of Lemma \ref{ops} implies there is an operator $E$, unitarily equivalent to $\tilde{E}$, but with diagonal $\{d_{i}\}_{i\in I}$. This completes the final case.
\end{proof}

We are now in a position to prove Theorem \ref{3pt}. In fact we will prove the following more general theorem.

\begin{thm}\label{3ptg} Let $0<A<B<\infty$ and let $\{d_{i}\}_{i\in I}$ be a sequence in $[0,B]$. If there is a positive operator $E$ with diagonal $\{d_{i}\}_{i\in I}$ and $\sigma(E) = \{0,A,B\}$ then one following holds:
\begin{enumerate}
\item $C=\infty$, 
\item $D=\infty$,
\item $C,D<\infty$ and there exist $N\in\N$ and $k\in\Z$ such that \eqref{3pttrace} and \eqref{3ptmaj} hold.
\end{enumerate}
Conversely, if $\sum d_{i} = \sum (B-d_{i}) = \infty$ and one of (i),(ii), or (iii) holds, then there is a positive operator $E$ with diagonal $\{d_{i}\}_{i\in I}$ and $\sigma(E) = \{0,A,B\}$.
\end{thm}

\begin{proof} First, assume that $E$ is a positive operator with spectrum $\{0,A,B\}$ and diagonal $\{d_{i}\}$. If either $C=\infty$ or $D=\infty$ then we are done since this is exactly (i) or (ii). If $C,D<\infty$ then Theorem \ref{N<infty} shows that \eqref{3pttrace} and \eqref{3ptmaj} hold and thus (iii) holds. 

Next, assume $\{d_{i}\}$ is a sequence in $[0,B]$. If (i) or (ii) holds then Theorem \ref{suff2} shows that there is a positive operator $E$ with spectrum $\{0,A,B\}$ and diagonal $\{d_{i}\}$. Finally, if (iii) holds and $\sum d_{i} = \sum (B-d_{i})=\infty$ then Theorem \ref{N<infty} shows that there is a positive operator $E$ with spectrum $\{0,A,B\}$ and diagonal $\{d_{i}\}$.
\end{proof}

\begin{remark}\label{rmk1}
In Theorem \ref{3pt} (and Theorem \ref{3ptg}) the assumption that $\sum d_{i} = \sum (B-d_{i})=\infty$ is necessary. Consider the sequence $\{A,0,0,\ldots\}$. This is clearly not the diagonal of any operator with spectrum $\{0,A,B\}$ since the operator would be trace class with trace equal to $A$, and thus $B>A$ cannot be an eigenvalue. However, we have $C = 0$ and $D=B-A$ so that \eqref{3pttrace} and \eqref{3ptmaj} hold with $N=1$ and $k=-1$.
\end{remark}

\begin{remark}\label{rmk2} 
There exist two non-unitarily equivalent operators with three point spectrum and the same diagonal. Let $0<A<B$ and let $I_{n}$ be the identity operator of an $n$ dimensional Hilbert space. From Theorem \ref{Kadison}, there is a projection $P$ with infinite dimensional kernel and range such that the diagonal of $BP$ consists of a countable infinite sequence of $A$'s. The operator $BP\oplus AI_{n}$ has a diagonal consisting of a countable number of $A$'s, however the multiplicity of the eigenvalue $A$ is $n$.
\end{remark}

\section{Outer eigenvalue with finite multiplicity}

In the last two remaining cases ((e) and (f)) of Theorem \ref{fullthm} we consider operators with finite dimensional kernel. In these cases, where there is an ``outer'' eigenvalue with finite multiplicity, we have the following necessary condition.

\begin{thm}\label{nec6} Let $0<A<B<\infty$ and let $E$ be a positive operator on a Hilbert space $\Hil$ with $\sigma(E)=\{0,A,B\}$ and $m_{E}(0)<\infty$. Let $\{e_{i}\}_{i\in I}$ be an orthonormal basis for $\Hil$ and set $d_{i}=\langle Ee_{i},e_{i}\rangle$. We have
\begin{equation}\label{nec6.1}\sum_{d_{i}<A}(A-d_{i})\leq A m_{E}(0).\end{equation}
\end{thm}

\begin{proof} There exist mutually orthogonal projections $P$ and $Q$ such that $E=AP+BQ$. Note that $I-P-Q$ is a finite rank projection and thus has finite trace equal to $m_{E}(0)$. Set $J_{1} = \{i\in I:d_{i}<A\}$. Then
\begin{align*}
\sum_{i\in J_{1}}(A-d_{i}) & = \sum_{i\in J_{1}}\big(A - A\langle Pe_{i},e_{i}\rangle - B\langle Qe_{i},e_{i}\rangle\big) \leq \sum_{i\in J_{1}}\big(A - A\langle Pe_{i},e_{i}\rangle - A\langle Qe_{i},e_{i}\rangle\big)\\
 & = A\left(\sum_{i\in J_{1}}\big(1 - \langle Pe_{i},e_{i}\rangle - \langle Qe_{i},e_{i}\rangle\big)\right) \leq A\left(\sum_{i\in I}\big(1 - \langle Pe_{i},e_{i}\rangle - \langle Qe_{i},e_{i}\rangle\big)\right)\\
 & = Am_{E}(0).
\end{align*}
\end{proof}

Next, we look at two examples which demonstrate that for operators with finite dimensional kernel the constants $C$ and $D$ do not capture enough information about a sequence in order to tell if it is the diagonal of an operator of the specified type.

\begin{ex}
Consider the sequence $\{d_{i}\}$ consisting of $\{1-i^{-1}\}_{i=1}^{\infty}$ and a countable infinite number of $2$'s. If $A=1$ and $B=2$ then we have $C=\infty$ and $D=0$. By Theorem \ref{nec6} this is not the diagonal of any positive operator $E$ with $\sigma(E)=\{0,1,2\}$ and finite dimensional kernel, since
\[\sum_{d_{i}<A}(A-d_{i}) = \sum_{i=1}^{\infty}\frac{1}{i}=\infty.\]
\end{ex}

\begin{ex}
Consider the sequence $\{c_{i}\}$ consisting of $\{1-2^{-i}\}_{i=1}^{\infty}$ and a countable infinite number of $2$'s. If $A=1$ and $B=2$ then we have $C=\infty$ and $D=0$.
By Theorem \ref{Kadison} there is a projection $P$ with diagonal $\{1-2^{-i}\}_{i=1}^{\infty}$ and finite dimensional kernel. Let $I$ be the identity operator on an infinite dimensional Hilbert space and set $E=P\oplus 2I$. This operator has diagonal $\{c_{i}\}$, spectrum $\{0,1,2\}$ and finite dimensional kernel. Note that $\{c_{i}\}$ and $\{d_{i}\}$ have the same values for $C$ and $D$, but only $\{c_{i}\}$ is the diagonal of an operator with spectrum $\{0,1,2\}$ and finite dimensional kernel.
\end{ex}

Instead of $C$ and $D$ we will use the following terminology from Theorem \ref{fullthm} in the rest of the section: 
\[
J_{1} = \{i:d_{i}<A\},\ J_{2} = \left\{i:d_{i}\in\left[A,\frac{A+B}{2}\right)\right\},\ J_{3} = \left\{i:d_{i}\geq \frac{A+B}{2}\right\}
\]

\[C_{1} = \sum_{i\in J_{1}}(A-d_{i}),\ C_{2} = \sum_{i\in J_{2}}(d_{i}-A),\ C_{3} = \sum_{i\in J_{3}}(B-d_{i})\]
Note that for symmetry we will use the notation $J_{1}$ instead of $I_{1}$, though they denote the same set.

The next theorem shows the necessity of the conditions in part (e) of Theorem \ref{fullthm}.

\begin{thm}\label{nec7} Let $0<A<B<\infty$ and let $E$ be a positive operator on a Hilbert space $\Hil$ with $\sigma(E)=\{0,A,B\}$, $m_{E}(0)<\infty$, and $m_{E}(A)=m_{E}(B)=\infty$. Let $\{e_{i}\}_{i\in I}$ be an orthonormal basis for $\Hil$ and set $d_{i} = \langle Ee_{i},e_{i}\rangle$.  If $C_{2},C_{3}<\infty$, then $C_{1}<\infty$, $|J_{1}\cup J_{2}|=|J_{3}|=\infty$, and there exist $n,k\in\Z$ such that $n+k=m_{E}(0)$,
\begin{equation}\label{nec7.1}C_{1}-C_{2}+C_{3} = nA + kB,\end{equation}
and
\begin{equation}\label{nec7.2}C_{1}\leq A(n+k).\end{equation}

\end{thm}

\begin{proof} There exist mutually orthogonal projections $P$ and $Q$ such that $E = AP+BQ$. Define $p_{i} = \langle Pe_{i},e_{i}\rangle$ and $q_{i} = \langle Qe_{i},e_{i}\rangle$ for each $i\in I$. Since $m_{E}(0)<\infty$, Theorem \ref{nec6} implies that $C_{1}\leq Am_{E}(0)<\infty$. Next, we note that
\begin{equation}\label{nec7.3}\sum_{i\in I}(1-p_{i}-q_{i}) = m_{P+Q}(0) = m_{E}(0) <\infty.\end{equation}
Using \eqref{nec7.3} we have
\begin{align*}
\sum_{i\in J_{1}\cup J_{2}}q_{i} & = \frac{1}{B-A}\left(\sum_{i\in J_{1}\cup J_{2}}(A-Ap_{i}-Aq_{i}) - \sum_{i\in J_{1}}(A-Ap_{i}-Bq_{i}) + \sum_{i\in J_{2}}(Bq_{i} + Ap_{i} - A)\right)\\
 & = \frac{1}{B-A}\left(\sum_{i\in J_{1}\cup J_{2}}(A-Ap_{i}-Aq_{i}) - C_{1} + C_{2}\right) \leq \frac{Am_{E}(0) - C_{1}+C_{2}}{B-A}<\infty.
\end{align*}
Together with \eqref{nec7.3} this also shows that $\sum_{i\in J_{1}\cup J_{2}}(1-p_{i})<\infty$. A similar calculation shows that
\[\sum_{i\in J_{3}}(1-q_{i}),\sum_{i\in J_{3}}p_{i}<\infty.\]
By Theorem \ref{Kadison} there exist $n,k\in\Z$ such that
\begin{equation}\label{eq2}\begin{split}
n = \sum_{i\in J_{1}\cup J_{2}}(1-p_{i}) - \sum_{i\in J_{3}}p_{i}\\
k = \sum_{i\in J_{3}}(1-q_{i}) - \sum_{i\in J_{1}\cup J_{2}} q_{i}.\\
\end{split}\end{equation}
Now, we calculate
\begin{align*}
C_{1}-C_{2}+C_{3} &  = \sum_{i\in J_{1}}(A-Ap_{i}-Bq_{i}) - \sum_{i\in J_{2}}(Ap_{i}+Bq_{i}-A) + \sum_{i\in J_{3}}(B-Ap_{i}-Bq_{i})\\
 & = A\sum_{i\in J_{1}\cup J_{2}}(1-p_{i}) - A\sum_{i\in J_{3}}p_{i} + B\sum_{i\in J_{3}}(1-q_{i}) - B\sum_{i\in J_{1}\cup J_{2}} q_{i}\\
 & = nA + kB,
\end{align*}
which shows \eqref{nec7.1} holds.

From \eqref{eq2} we have
\[n+k = \sum_{i\in I}(1-p_{i}-q_{i}) = m_{E}(0).\]
Theorem \ref{nec6} shows $C_{1}\leq Am_{E}(0)=A(n+k)$, which is \eqref{nec7.2}.

Note that
\[\sum_{i\in I}p_{i} = \dim\ran P = m_{E}(A) = \infty.\]
Since $\sum_{i\in J_{3}}p_{i}<\infty$, it must be the case that $\sum_{i\in J_{1}\cup J_{2}}p_{i}=\infty$ and thus $|J_{1}\cup J_{2}|=\infty$. Similarly, since $Q$ has infinite dimensional range, we have $\sum_{i\in I}q_{i}=\infty$. Since $\sum_{i\in J_{1}\cup J_{2}}q_{i}<\infty$ it must be the case that $\sum_{i\in J_{3}}q_{i} = \infty$, and thus $|J_{3}|=\infty$.
\end{proof}

The next theorem shows that the conditions in part (e) of Theorem \ref{fullthm} are sufficient to construct the desired operator. We state it in a slightly more general form for use in the proof of part (f) later in this section.

\begin{thm}\label{suff4} Let $0<A<B<\infty$, let $\{d_{i}\}_{i\in I}$ be a sequence in $[0,B]$, and let $Z\in\N$. If $|J_{1}\cup J_{2}|=\infty$,
$C_{1}\leq AZ$, and either of the following holds:
\begin{enumerate}
\item $C_{2}+C_{3}=\infty$
\item $C_{2},C_{3}<\infty$ and there exists $n,k\in\Z$ such that $Z=n+k$ and
\begin{equation}\label{trace}C_{1}-C_{2}+C_{3}=nA+kB,\end{equation}
\end{enumerate}
then there is a positive operator $E$ with $\sigma(E)=\{0,A,B\}$, $m_{E}(0)=Z$, $m_{E}(A)=\infty$, and diagonal $\{d_{i}\}$. Moreover, if (i) holds then $m_{E}(B)=\infty$, and if (ii) holds then $m_{E}(B)=|J_{3}|-k$.
\end{thm}

\begin{proof} Set
\[\eta = AZ - C_{1}.\]

{\bf{Case 1:}} Assume
\[\sum_{i\in J_{1}}d_{i},\sum_{i\in J_{2}\cup J_{3}}(B-d_{i})>\eta.\]
There are finite subsets $F_{0}\subset J_{1}$ and $F_{1}\subset J_{2}\cup J_{3}$ such that
\[\eta \leq \min\left\{\sum_{i\in F_{0}}d_{i},\sum_{i\in F_{1}}(B-d_{i})\right\}.\]
We can apply Lemma \ref{ops} (i) with $F_{0}$ and $F_{1}$ as above, and $\eta_{0} = \eta$, to obtain $\{\tilde{d}_{i}\}_{i\in I}$. From \eqref{ops2} we have
\begin{align*}
\sum_{i\in J_{1}}(A-\tilde{d}_{i}) & = \sum_{i\in F_{0}}(A-\tilde{d}_{i}) + \sum_{i\in J_{1}\setminus F_{0}}(A-d_{i}) = |F_{0}|A - \sum_{i\in F_{0}}\tilde{d}_{i} + \sum_{i\in J_{1}\setminus F_{0}}(A-d_{i})\\
 & = |F_{0}|A + \eta - \sum_{i\in F_{0}}d_{i} + \sum_{i\in J_{1}\setminus F_{0}}(A-d_{i}) = \eta + \sum_{i\in J_{1}}(A-d_{i}) = \eta + C_{1} = AZ.
\end{align*}
Theorem \ref{Kadison} implies there is a projection $P$ with $Z$ dimensional kernel such that $AP$ has diagonal $\{\tilde{d}_{i}\}_{i\in J_{1}}$. It is clear that if $|J_{1}|=\infty$ then $m_{P}(1)=\infty$.

If (i) holds, that is $C_{2}+C_{3}=\infty$, then Theorem \ref{Kadison} implies there is a projection $Q_{1}$ such that $(B-A)Q_{1}$ has diagonal $\{\tilde{d}_{i}-A\}_{i\in J_{2}\cup J_{3}}$. Since
\[\sum_{i\in J_{2}\cup J_{3}}(\tilde{d}_{i}-A) = \sum_{i\in J_{2}\cup J_{3}}\big((B-A) - (\tilde{d}_{i}-A)\big) = \infty,\]
we also see that $m_{Q_{1}}(0)=m_{Q_{1}}(1)=\infty$. Set $\tilde{E} = AP\oplus\big((B-A)Q_{1}+AI\big)$. It is clear that $m_{\tilde{E}}(0)=Z$, $\sigma(\tilde{E}) = \{0,A,B\}$, and $m_{\tilde{E}}(A) = m_{\tilde{E}}(B)=\infty$. By the second part of Lemma \ref{ops} there is an operator $E$, unitarily equivalent to $\tilde{E}$, with diagonal $\{d_{i}\}_{i\in I}$.

If (ii) holds, then using \eqref{trace} we have
\begin{align*}
\sum_{i\in J_{2}}(\tilde{d}_{i} - A) - \sum_{i\in J_{3}}(B-\tilde{d}_{i}) & = \eta + \sum_{i\in J_{2}}(d_{i} - A) - \sum_{i\in J_{3}}(B-d_{i}) = \eta + C_{2}-C_{3}\\
 & = AZ - C_{1} + C_{1} - An - Bk = -k(B-A).
\end{align*}
Theorem \ref{Kadison} implies there is a projection $Q_{2}$ such that $(B-A)Q_{2}$ has diagonal $\{\tilde{d}_{i}-A\}_{i\in J_{2}\cup J_{3}}$. The operator $\tilde{E} = AP\oplus\big((B-A)Q_{2}+AI\big)$ has diagonal $\{\tilde{d}_{i}\}_{i\in I}$, and it is clear that $m_{\tilde{E}}(0) = Z$ and $\sigma(\tilde{E}) = \{0,A,B\}$. Note that if $|J_{2}|=\infty$ then $m_{Q_{2}}(0) = \infty$ and we already noted that $|J_{1}|=\infty$ implies $m_{P}(1)=\infty$; in either case $m_{\tilde{E}}(A) = \infty$. If $|J_{3}|=\infty$ we have $m_{Q_{2}}(1)=\infty$ and thus $m_{\tilde{E}}(B) = \infty$. If $|J_{3}|<\infty$ then we have
\[\sum_{i\in J_{2}\cup J_{3}}(\tilde{d}_{i}-A) = \sum_{i\in J_{2}}(\tilde{d}_{i} - A) - \sum_{i\in J_{3}}(B-\tilde{d}_{i}) +|J_{3}|(B-A) = (|J_{3}|-k)(B-A),\]
which implies $m_{Q_{2}}(1) = |J_{3}|-k$, and thus $m_{\tilde{E}}(B) = |J_{3}|-k$. By Lemma \ref{ops} (ii), there is an operator $E$, unitarily equivalent to $\tilde{E}$, with diagonal $\{d_{i}\}$. This completes the proof of Case 1.

{\bf{Case 2:}} Assume
\[\sum_{i\in J_{1}}d_{i}\leq\eta.\]
This implies $J_{1}$ is a finite set and that $|J_{1}|\leq Z$. Since $|J_{1}\cup J_{2}|=\infty$ this implies $|J_{2}|=\infty$, and thus
\[\sum_{i\in J_{2}}(B-d_{i})=\infty.\]
Let $L,F_{1}\subset J_{2}\cup J_{3}$ be disjoint finite sets which satisfy three conditions:
\[\sum_{i\in F_{1}}(B-d_{i})>BZ,\]
$|L|=Z-|J_{1}|$, and $\max\{d_{i}\}_{i\in L}\leq\min\{d_{i}\}_{i\in F_{1}}$. Set $F_{0} = J_{1}\cup L$. Apply Lemma \ref{ops} (i) with $F_{0}$ and $F_{1}$ as already defined, and
\[\eta_{0} = \sum_{i\in F_{0}}d_{i}<BZ\]
to obtain the sequence $\{\tilde{d}_{i}\}_{i\in I}$. The choice of $\eta_{0}$ implies that $\{\tilde{d}_{i}\}_{i\in F_{0}}$ is a sequence of $Z$ zeroes.

If (i) holds, then we have
\[\sum_{i\in J_{2}}(\tilde{d}_{i}-A) + \sum_{i\in J_{3}}(B-\tilde{d}_{i}) = \infty.\]
Theorem \ref{Kadison} implies that there is a projection $Q_{1}$ such that $(B-A)Q_{1}$ has diagonal $\{\tilde{d}_{i}-A\}_{i\in J_{2}\cup J_{3}}$ and $m_{Q_{1}}(0)=m_{Q_{1}}(1)=\infty$. Let ${\bf 0}_{Z}$ be the zero operator on a $Z$ dimensional Hilbert space, and set $\tilde{E} = {\bf 0}_{Z}\oplus\big((B-A)Q_{1}+AI\big)$. It is clear that $\tilde{E}$ has diagonal $\{\tilde{d}_{i}\}$, $m_{\tilde{E}}(0)=Z$, $\sigma(\tilde{E}) = \{0,A,B\}$, and $m_{\tilde{E}}(A) = m_{\tilde{E}}(B)=\infty$. By Lemma \ref{ops} (ii), there is an operator $E$, unitarily equivalent to $\tilde{E}$, with diagonal $\{d_{i}\}_{i\in I}$.

If (ii) holds then by \eqref{trace} we have
\begin{align*}
\sum_{i\in J_{2}\setminus L}(\tilde{d}_{i} - A) - \sum_{i\in J_{3}}(B-\tilde{d}_{i}) & = \eta_{0} + \sum_{i\in J_{2}\setminus L}(d_{i} - A) - \sum_{i\in J_{3}}(B-d_{i})\\
 & = \sum_{i\in J_{1}}d_{i} + \sum_{i\in L}d_{i} + \sum_{i\in J_{2}\setminus L}(d_{i} - A) - C_{3}\\
 & = -C_{1} + C_{2} - C_{3} + (|J_{1}|+|L|)A\\
 & = -nA-kB + ZA = -k(B-A).
\end{align*}
Theorem \ref{Kadison} implies there is a projection $Q_{2}$ such that $(B-A)Q_{2}$ has diagonal $\{\tilde{d}_{i}-A\}_{i\in (J_{2}\cup J_{3})\setminus L}$. Since $J_{2}$ is infinite we have $m_{Q_{2}}(0)=\infty$. If $J_{3}$ is infinite then we also have $m_{Q_{2}}(1)=\infty$. If $|J_{3}|<\infty$ then
\[\sum_{i\in(J_{2}\cup J_{3})\setminus L}(\tilde{d}_{i}-A) = \sum_{i\in J_{2}\setminus L}(\tilde{d}_{i} - A) - \sum_{i\in J_{3}}(B-\tilde{d}_{i}) +|J_{3}|(B-A) = (|J_{3}|-k)(B-A),\]
which implies $m_{Q_{2}}(1) = |J_{3}|-k$. The operator $\tilde{E} = {\mathbf 0}_{Z}\oplus\big((B-A)Q_{2}+AI\big)$ has the desired eigenvalues and multiplicities and diagonal $\{\tilde{d}_{i}\}$. Lemma \ref{ops} (ii) implies there is an operator $E$, unitarily equivalent to $\tilde{E}$, with diagonal $\{d_{i}\}$. This completes the proof of the second case.

{\bf Case 3:} Assume
\[\sum_{i\in J_{2}\cup J_{3}}(B-d_{i})\leq \eta.\]
This implies $J_{2}$ is finite, since $d_{i}<(B+A)/2$ for all $i\in J_{2}$. By hypothesis $|J_{1}\cup J_{2}|=\infty$, and thus  $J_{1}$ must be infinite. Moreover, $A$ is a limit point of $\{d_{i}\}_{i\in J_{1}}$, since $\sum_{i\in J_{1}}(A-d_{i})<\infty$ and $d_{i}<A$ for all $i\in J_{1}$. There is some $N_{0}\in\N$ such that
\[(B-A)N_{0}>\eta.\]
Choose $\alpha\in(0,A)$ such that
\[\sum_{d_{i}<\alpha}d_{i}>AN_{0}.\]
Set $F_{0} = \{i\in J_{1}:d_{i}<\alpha\}$, and note that it is finite since $C_{1}<\infty$. Since $A$ is a limit point of $\{d_{i}\}_{i\in J_{1}}$, we can find a set $F_{1}\subset \{i\in J_{1}:d_{i}\geq\alpha\}$ with $N_{0}$ elements, and clearly 
\[\sum_{i\in F_{1}}(A-d_{i})<AN_{0}.\]
Applying Lemma \ref{ops} (i) on the interval $[0,A]$, with $F_{0}$ and $F_{1}$ as above, and
\[\eta_{0} = \sum_{i\in F_{1}}(A-d_{i}),\]
we obtain a sequence $\{\tilde{d}_{i}\}_{i\in I}$. Using \eqref{ops2} we see that $\tilde{d}_{i}=A$ for each $i\in F_{1}$. We also have
\[\sum_{i\in F_{0}}(A-\tilde{d}_{i}) = |F_{0}|A - \sum_{i\in F_{0}}\tilde{d}_{i} = |J_{1}|A - \sum_{i\in F_{0}}d_{i} - \sum_{i\in F_{1}}(A-d_{i}) = \sum_{i\in F_{0}\cup F_{1}}(A-d_{i}).\]
Define the sets
\[
\tilde{J}_{1} = \{i:\tilde{d}_{i}<A\},\ \tilde{J}_{2} = \left\{i:\tilde{d}_{i}\in\left[A,\frac{A+B}{2}\right)\right\},\ \tilde{J}_{3} = \left\{i:\tilde{d}_{i}\geq \frac{A+B}{2}\right\}.
\]
We have
\begin{align*}
\sum_{i\in\tilde{J}_{1}}(A-\tilde{d}_{i}) & = \sum_{i\in J_{1}\setminus(F_{0}\cup F_{1})}(A-d_{i}) + \sum_{i\in F_{0}}(A-\tilde{d}_{i}) = \sum_{i\in J_{1}\setminus(F_{0}\cup F_{1})}(A-d_{i}) + \sum_{i\in F_{0}\cup F_{1}}(A-d_{i})\\
 & = \sum_{i\in J_{1}}(A-d_{i}) = C_{1}.
\end{align*}
Since $\tilde{d}_{i}=A$ for all $i\in F_{1}$ we have
\[\sum_{i\in\tilde{J}_{2}}(\tilde{d}_{i}-A) = \sum_{i\in J_{2}}(d_{i}-A) + \sum_{i\in F_{1}}(\tilde{d}_{i}-A) = \sum_{i\in J_{2}}(d_{i}-A) = C_{2}.\]
Lastly, $\tilde{d}_{i}=d_{i}$ for all $i\in J_{3}$, and thus
\[\sum_{i\in\tilde{J}_{3}}(B-\tilde{d}_{i}) = C_{3}.\]
However,
\[\sum_{i\in \tilde{J}_{2}\cup\tilde{J}_{3}}(B-\tilde{d}_{i}) = \sum_{i\in J_{2}\cup J_{3}}(B-d_{i}) + (B-A)N_{0}>\eta.\]
This implies that $\{\tilde{d}_{i}\}_{i\in I}$ satisfies the conditions of Case 1, and thus there is an operator $\tilde{E}$ with the desired eigenvalues and multiplicities and diagonal $\{\tilde{d}_{i}\}_{i\in I}$. By Lemma \ref{ops} (ii), there is an operator $E$, unitarily equivalent to $\tilde{E}$, with diagonal $\{d_{i}\}_{i\in I}$. This completes the proof of this case and the proof of the theorem.
\end{proof}

As a corollary of Theorems \ref{nec7} and \ref{suff4} we deduce part (f) of Theorem \ref{fullthm}. This will complete the proof of Theorem \ref{fullthm}.

\begin{cor}\label{part(f)} Let $0<A<B<\infty$, let $\{d_{i}\}_{i\in I}$ be a sequence in $[0,B]$, and let $Z,K\in\N$. There exists a positive operator $E$ with $\sigma(E) = \{0,A,B\}$, $m_{E}(0)=Z$, $m_{E}(A) = \infty$, $m_{E}(B) = K$ and diagonal $\{d_{i}\}$ if and only if $|I|=\infty$, $C_{1}\leq ZA$ and
\begin{equation}\label{trace4}\sum_{i\in I}(d_{i}-A) = K(B-A) - ZA.\end{equation}
\end{cor}

\begin{proof} First, assume that $|I|=\infty$, $C_{1}\leq ZA$ and \eqref{trace4} holds. It is clear that $|J_{3}|<\infty$ and thus $|J_{1}\cup J_{2}| = |I\setminus J_{3}|=\infty$. We have
\[C_{1}-C_{2}+C_{3} = - \sum_{i\in I}(d_{i}-A) + |J_{3}|(B-A) = (Z+K-|J_{3}|)A + (|J_{3}|-K)B.\]
By Theorem \ref{suff4} the desired operator exists.

Next, assume the operator $E$ exists. Note that $E-A$ is a finite rank operator, and thus it is of trace class with trace
\[\sum_{i\in I}(d_{i}-A) = K(B-A) - AZ.\]
By Theorem \ref{nec6} we have $C_{1}\leq ZA$. Since $m_{E}(A)=\infty$, the operator $E$ is acting on an infinite dimensional Hilbert space, thus $|I|=\infty$.
\end{proof}

\section{Examples}

To demonstrate the use of Theorem \ref{3pt} we will consider the following problem: Given a sequence $\{d_{i}\}$ in $[0,1]$, for what values of $A$ is there a positive operator $E$ with $\sigma(E) = \{0,A,1\}$ and diagonal $\{d_{i}\}$? First, we will prove the following general theorem.

\begin{thm}\label{exthm} Let $\{d_{i}\}_{i\in\N}$ be a sequence in $[0,1]$ and set
\[\mathcal{A} = \big\{A\in(0,1):\exists\,E\geq 0\text{ with } \sigma(E)=\{0,A,1\}\text{ and diagonal }\{d_{i}\}\big\}.\]
Either $\mathcal{A} = (0,1)$ or $\mathcal{A}$ is a finite (possibly empty) set.
\end{thm}

\begin{proof} For each $A\in (0,1)$ define
\[C(A) = \sum_{d_{i}<A}d_{i}\quad\text{and}\quad D(A) = \sum_{d_{i}\geq A}(1-d_{i}).\]
Note that if $C(A)+D(A)=\infty$ for some $A\in(0,1)$ then $C(A)+D(A)=\infty$ for all $A\in(0,1)$. By Theorem \ref{3ptg} we have $\mathcal{A}=(0,1)$. Thus, we will assume $C(A),D(A)<\infty$ for all $A\in (0,1)$.

First, we wish to show that $\sup\mathcal{A}<1$. Assume to the contrary that $\sup\mathcal{A}=1$. Note that there exists $\eta\in[0,1)$ such that $\eta = C(A)-D(A) - \lfloor C(A)-D(A)\rfloor$ for all $A\in (0,1)$. Thus, for each $A\in(0,1)$ there exists $m(A)\in\Z$ such that
\[C(A)-D(A) = m(A) +\eta.\]
By Theorem \ref{3ptg}, for each $A\in\mathcal{A}$ there exists $N(A)\in\N$ and $k(A)\in\Z$ such that
\begin{equation}\label{exthm5}m(A) + \eta = C(A)-D(A) = N(A)A+k(A)\quad\text{and}\quad C(A)\geq (N(A)+k(A))A.\end{equation}
Using \eqref{exthm5} we have
\begin{equation}\label{exthm4}m(A) + \eta = N(A)A+k(A)< N(A)+k(A)\leq \frac{C(A)}{A}.\end{equation}
Since $\eta\geq 0$ and $m(A),N(A),k(A)\in\Z$, we can also see
\begin{equation}\label{exthm3}m(A)+1\leq N(A)+k(A).\end{equation}
Thus, for each $A\in\mathcal{A}$ we must have
\begin{equation}\label{exthm1}A(m(A) + 1)\leq C(A).\end{equation}
Next, note that for $A,A'\in\mathcal{A}$ with $A'>A$ we have
\begin{align*}
m(A')-m(A) & = C(A') - C(A) + D(A) - D(A') = \sum_{A\leq d_{i}<A'}d_{i} + \sum_{A\leq d_{i}<A'}(1-d_{i})\\
 & = |\{i\in\N: A\leq d_{i}<A'\}|.
\end{align*}
Using this gives
\begin{equation}\label{exthm2}C(A')- C(A)= \sum_{d_{i}<A'}d_{i} - \sum_{d_{i}<A}d_{i} = \sum_{A\leq d_{i}<A'}d_{i}<  A'(m(A')-m(A)).\end{equation}
Putting together \eqref{exthm1} and \eqref{exthm2} we have
\[A'(m(A')+1)-C(A)\leq C(A')-C(A) < A'(m(A')-m(A)).\]
Rearranging this inequality gives
\[A'(m(A)+1)< C(A).\]
Since $\sup\mathcal{A}=1$ we can let $A'\to 1$ and we have
\[m(A)+1\leq C(A).\]
Finally, since $D(A)\to0$ as $A\to 1$, for large enough $A$ we have $D(A)<1-\eta$ and thus
\[C(A)<C(A)-D(A) - \eta + 1 = m(A) + 1\]
which gives a contradiction, and shows that $A_{\sup}:=\sup\mathcal{A}<1$. A symmetric argument shows that $A_{\inf}:=\inf\mathcal{A}>0$.

Since $C(A)$ and $m(A)$ are nondecreasing as $A\to 1$, for each $A\in\mathcal{A}$ we have $C(A_{\inf})\leq C(A)\leq C(A_{\sup})$ and $m(A_{\inf})\leq m(A)\leq m(A_{\sup})$. Using \eqref{exthm4} and \eqref{exthm3}, for $A\in\mathcal{A}$ we have
\[m(A_{\inf}) + 1 \leq m(A) + 1 \leq N(A)+k(A) \leq \frac{C(A)}{A}\leq \frac{C(A_{\sup})}{A_{\inf}}.\]
This shows that $\{N(A)+k(A):A\in\mathcal{A}\}$ and $\{m(A):A\in\mathcal{A}\}$ are finite sets of integers. Next, we note that for $A\in\mathcal{A}$ we have
\[N(A)A_{\sup}\geq N(A)A = m(A) + \eta - k(A) \geq m(A_{\inf}) + \eta + N(A) - \frac{C(A_{\sup})}{A_{\inf}}. \]
Rearranging this inequality gives
\[N(A) \leq \frac{\frac{C(A_{\sup})}{A_{\inf}} - m(A_{\inf}) - \eta}{1-A_{\sup}},\]
which implies that $\{N(A):A\in\mathcal{A}\}\subset\N$ is finite. Since $\{N(A)+k(A):A\in\mathcal{A}\}$ is finite, we also see that $\{k(A):A\in\mathcal{A}\}$ is finite. Finally, we note that for $A\in\mathcal{A}$ we have
\[A = \frac{m(A) + \eta - k(A)}{N(A)},\]
which clearly implies that $\mathcal{A}$ is finite. \end{proof}

Next, we will explicitly find the set $\mathcal{A}$ from Theorem \ref{exthm} for two particular sequences $\{d_{i}\}$.

\begin{ex} Let $\beta\in(0,1/2)$ and define the sequence $\{d_{i}\}_{i\in\Z\setminus\{0\}}$ by
\[d_{i} = \begin{cases}1-\beta^{i} & i>0\\ \beta^{-i} & i<0.\end{cases}\]
Define the set
\[\mathcal{A}_{\beta} = \big\{A\in (0,1):\exists\,E\geq0\text{ with } \sigma(E) = \{0,A,1\}\text{ and diagonal }\{d_{i}\}\big\}.\]
We will show that 
\[\mathcal{A}_{\beta} = \begin{cases}\{\frac{1}{3},\frac{1}{2},\frac{2}{3}\} & \frac{-1+\sqrt{13}}{6}\leq\beta<1/2\\ \{\frac{1}{2}\} & 1/3\leq\beta<\frac{-1+\sqrt{13}}{6}\\ \varnothing & 0<\beta<1/3.\end{cases}\]

First, assume $A\in \mathcal{A}_{\beta}\cap(\beta,1-\beta]$, and thus
\[C = \sum_{d_{i}<A}d_{i} = \sum_{i=1}^{\infty}\beta^{i} = \frac{\beta}{1-\beta}\quad\text{and}\quad D = \sum_{d_{i}\geq A}(1-d_{i}) = \sum_{i=1}^{\infty}\beta^{i} = \frac{\beta}{1-\beta}.\]
From Theorem \ref{3ptg} there exists $N\in \N$ and $k\in\Z$ such that
\begin{equation}\label{ex1} 0 = C-D = NA+k\end{equation}
\begin{equation}\label{ex2} \frac{\beta}{1-\beta} = C \geq (N+k)A.\end{equation}
Using \eqref{ex1} and $A\leq 1-\beta$ we have
\[0 < \beta N = NA+k+\beta N \leq N(1-\beta)+k + \beta N=N+k,\]
and thus $N+k>0$. Now, we use \eqref{ex2}, $\beta<A$, then $\beta<1/2$ to see
\[ N+k < (N+k)\frac{A}{\beta} \leq \beta^{-1}\frac{\beta}{1-\beta} = \frac{1}{1-\beta} < 2.\]
Since $N+k\in\Z$ we see that $N+k=1$. Solving for $A$ in \eqref{ex1} we have
\[A = \frac{-k}{N} = \frac{N-1}{N} = 1-\frac{1}{N}.\]
Since $\{d_{i}\}$ is symmetric about $1/2$, if $A\in\mathcal{A}_{\beta}$ then $1-A\in\mathcal{A}_{\beta}$. And since $A=1-1/N$ for some $N\in\N$, the only possible value of $N$ is $2$. For $N=2$ we have $k=-1$ and $A=1/2$, which satisfy \eqref{3pttrace} and \eqref{3ptmaj} if and only if $\beta\geq 1/3$. Thus, $\mathcal{A}_{\beta}\cap(\beta,1-\beta]=\{1/2\}$ for $\beta\geq 1/3$ and $\mathcal{A}_{\beta}\cap(\beta,1-\beta]=\varnothing$ for $\beta<1/3$.

Next, assume $A\in\mathcal{A}_{\beta}\cap(1-\beta^{m},1-\beta^{m+1}]$ for some $m\in\N$. We have
\[C = \frac{\beta}{1-\beta} + \sum_{i=1}^{m}(1-\beta^{i}) = m + \frac{\beta^{m+1}}{1-\beta}\quad\text{and}\quad D =\sum_{i=m+1}^{\infty}\beta^{i} = \frac{\beta^{m+1}}{1-\beta}.\]
By Theorem \ref{3ptg} there exist $N\in\N$ and $k\in\Z$ such that
\begin{equation}\label{ex3} m = C-D = NA + k\end{equation}
\begin{equation}\label{ex4} m + \frac{\beta^{m+1}}{1-\beta} = C \geq (N+k)A.\end{equation}
Using \eqref{ex3} and $A\leq 1-\beta^{m+1}$ we have
\begin{equation}\label{ex5}m < m+N\beta^{m+1}= NA+k + N\beta^{m+1}\leq N(1-\beta^{m+1}) + k + N\beta^{m+1} = N+k.\end{equation}
Using \eqref{ex4} and $A>1-\beta^{m}$ we have
\[m + \frac{\beta^{m+1}}{1-\beta}\geq (N+k)A>(N+k)(1-\beta^{m}).\]
Rearranging, and using $\beta<1/2$ we have
\begin{equation}\label{ex6}N+k <\left(m + \frac{\beta^{m+1}}{1-\beta}\right)\frac{1}{1-\beta^{m}}<\left(m+\frac{1}{2^{m}}\right)\frac{2^{m}}{2^{m}-1} = m+\frac{1+m}{2^{m}-1}.\end{equation}
A simple calculation shows that $\frac{1+m}{2^{m}-1}\leq 1$ for all $m\geq 2$. Combining this with \eqref{ex5} shows that $m<N+k<m+1$ for $m\geq 2$. Since $N+k\in\Z$ this shows that $\mathcal{A}_{\beta}\cap(1-\beta^{2},1)=\varnothing$.

Assume $A\in(1-\beta,1-\beta^{2}]$. In this case \eqref{ex5} and \eqref{ex6} imply $1<N+k<3$, which implies $N+k=2$. Solving \eqref{ex3} for $A$ and using $N+k=2$ we have
\[A = \frac{1-k}{N} = 1-\frac{1}{N}.\]
Since $A>1-\beta>1/2$ this implies $N>1/\beta>2$. From \eqref{ex4} we see
\[1+\frac{\beta^{2}}{1-\beta}\geq 2A = 2-\frac{2}{N}.\]
Rearranging this we have
\[N\leq \frac{2-2\beta}{1-\beta-\beta^{2}}.\]
For $\beta<\frac{-1+\sqrt{13}}{6}$ we have $\frac{2-2\beta}{1-\beta-\beta^{2}}<3$ and thus $N<3$. Combined with the fact that $N>2$, we see $\mathcal{A}\cap(1-\beta,1-\beta^{2}]=\varnothing$ for $\beta<\frac{-1+\sqrt{13}}{6}$. Next, assume $\frac{-1+\sqrt{13}}{6}\leq\beta<1/2$. Then $\frac{2-2\beta}{1-\beta-\beta^{2}}<4$ and we must have $N=3$, $A=\frac{2}{3}$ and $k=-1$. It is clear that \eqref{3pttrace} holds. For \eqref{3ptmaj}, we use the fact that $\beta\geq\frac{-1+\sqrt{13}}{6}$ to see
\[C = 1+\frac{\beta^{2}}{1-\beta}\geq \frac{4}{3} = (N+k)A.\]
Thus, by Theorem \ref{3pt}, for $\beta\geq\frac{-1+\sqrt{13}}{6}$ we have $2/3\in\mathcal{A}_{\beta}$. Since $\{d_{i}\}$ is symmetric about $1/2$, we see that $\mathcal{A}_{\beta}\cap(0,\beta] = \{1/3\}$ for $\frac{-1+\sqrt{13}}{6}\leq \beta<1/2$ and the set is empty for $\beta<\frac{-1+\sqrt{13}}{6}$. $\Box$

\end{ex}

In the above example, note that for any choice of $\beta$, we have $C-D\in\Z$ for any choice of $A\in (0,1)$. Thus, Theorem \ref{Kadison} implies that there is a projection with diagonal $\{d_{i}\}$. However, if $\beta<1/3$ then there is no $A\in(0,1)$ so that $\{d_{i}\}$ is the diagonal of a self-adjoint operator $E$ with $\sigma(E) = \{0,A,1\}$. The next example is not the diagonal of any projection, but we will show that it is the diagonal of many different operators with three point spectrum.

\begin{ex} Let $\{d_{i}\}_{i\in\Z}$ be given by 
\[d_{i} = \begin{cases} 2^{i-1} & i\leq 0\\ 1-2^{-i-1} & i>0.\end{cases}\]
Let
\[\mathcal{A} = \big\{A\in (0,1):\exists\,E\geq0\text{ with } \sigma(E) = \{0,A,1\}\text{ and diagonal }\{d_{i}\}\big\}.\]
We claim that
\[\mathcal{A} =\left\{\frac{1}{8},\frac{1}{6},\frac{1}{4},\frac{1}{2},\frac{3}{4},\frac{5}{6},\frac{7}{8}\right\}.\]
The sequence $\{d_{i}\}$ is symmetric about $1/2$, and thus $A\in\mathcal{A}$ implies $1-A\in\mathcal{A}$. Hence, it is enough to show that
\[\mathcal{A}\cap\left[\tfrac{1}{2},1\right) = \left\{\frac{1}{2},\frac{3}{4},\frac{5}{6},\frac{7}{8}\right\}.\]

Assume $A\in\mathcal{A}\cap(1-2^{-m},1-2^{-m-1}]$ for some $m\geq 1$. We have
\[C = m - \frac{1}{2} + \frac{1}{2^{m}}\quad\text{and}\quad D = \frac{1}{2^{m}}.\]
Since $A\in\mathcal{A}$ Theorem \ref{3ptg} implies that there exist $N\in\N$ and $k\in\Z$ such that
\begin{equation}\label{ex2.1}C-D = m-\frac{1}{2} = NA+k\end{equation}
\begin{equation}\label{ex2.2}C = m-\frac{1}{2} + 2^{-m}\geq (N+k)A.\end{equation}
Using \eqref{ex2.1} and $A\leq 1-2^{-m-1}$ we have
\begin{equation}\label{ex2.3}m-1<m-\frac{1}{2} + N2^{-m-1} = NA + k + N2^{-m-1} \leq N(1-2^{-m-1}) + k + N2^{-m-1} = N+k.\end{equation}
From \eqref{ex2.2} and $A>1-2^{-m}$ we have
\[m-\frac{1}{2} + 2^{-m}\geq (N+k)A> (N+k)(1-2^{-m}).\]
Rearranging gives
\begin{equation}\label{ex2.4}N+k < \left(m-\frac{1}{2} + 2^{-m}\right)\frac{2^{m}}{2^{m}-1} = m + \frac{m-2^{m-1} + 1}{2^{m}-1}.\end{equation}
For $m\geq 3$, a simple calculation shows $\frac{m-2^{m-1} + 1}{2^{m}-1}\leq 0$ and thus $N+k<m$. However, from \eqref{ex2.3} we have $N+k>m-1$. Since $N+k\in\Z$ this is a contradiction and shows that $\mathcal{A}\cap(1-2^{-m},1-2^{-m-1}]=\varnothing$ for $m\geq 3$. 

One can easily check that $A=1/2$ satisfies \eqref{3pttrace} and \eqref{3ptmaj} with $N=1$ and $k=-1$ (or $N=3$ and $k=-2$). All that is left is to find $\mathcal{A}\cap(1-2^{-m},1-2^{-m-1}]$ for $m=1$ and $2$. The calculation for each $m$ is similar, so the case of $m=2$ will be left to the reader.

Assume $A\in\mathcal{A}\cap(1/2,3/4]$. In this case we have $C=1$ and $D=1/2$. From \eqref{ex2.3} and \eqref{ex2.4} we have $0<N+k<2$ and thus $N+k=1$. Using this and solving \eqref{ex2.1} for $A$ we have
\[A = \frac{\frac{1}{2}-k}{N} = \frac{N-\frac{1}{2}}{N} = 1-\frac{1}{2N}.\]
From the inequalities $1/2<A=1-1/(2N)\leq 3/4$ we obtain $1<N\leq 2$. Thus $N=2,A=3/4$ and $k=-1$. One can easily check that \eqref{3pttrace} and \eqref{3ptmaj} are satisfied for these values of $A,N$ and $k$. $\Box$
\end{ex}


\begin{thebibliography}{99}

\bibitem{amrs}
J. Antezana, P. Massey, M. Ruiz, D. Stojanoff,
{\it The Schur-Horn theorem for operators and frames with prescribed norms and frame operator},
Illinois J. Math. {\bf 51} (2007), 537--560.

\bibitem{am}
M. Argerami, P. Massey,
{\it A Schur-Horn theorem in ${\rm II}\sb 1$ factors}, 
Indiana Univ. Math. J. {\bf 56} (2007), 2051--2059.

\bibitem{am2}
M. Argerami, P. Massey,
{\it Towards the Carpenter's theorem}, Proc. Amer. Math. Soc. {\bf 137} (2009), 3679--3687.

\bibitem{a}
W. Arveson, 
{\it Diagonals of normal operators with finite spectrum},
Proc. Natl. Acad. Sci. USA {\bf 104} (2007), 1152--1158.

\bibitem{ak}
W. Arveson, R. Kadison, 
{\it Diagonals of self-adjoint operators}, Operator theory, operator algebras, and applications, 247--263, Contemp. Math., {\bf 414}, Amer. Math. Soc., Providence, RI, 2006.

\bibitem{bf}
J. Benedetto, M. Fickus,
{\it Finite normalized tight frames},
Adv. Comput. Math. 18 (2003), 357--385.

\bibitem{mbjj}
M.~Bownik, J.~Jasper,
{\it Characterization of sequences of frame norms},
J. Reine Angew. Math. (to appear).

\bibitem{cas}
P. Casazza, 
{\it Custom building finite frames},
Wavelets, frames and operator theory, 61--86, Contemp. Math., {\bf 345}, Amer. Math. Soc., Providence, RI, 2004.

\bibitem{cl}
P. Casazza, M. Leon,
{\it Existence and construction of finite tight frames},
J. Concr. Appl. Math. {\bf 4} (2006), 277--289.

\bibitem{cfklt}
P. Casazza, M. Fickus, J. Kova\v cevi\'c, M. Leon, J. Tremain,
{\it A physical interpretation of tight frames}. Harmonic analysis and applications, 51--76, Appl. Numer. Harmon. Anal., Birkh\"auser Boston, Boston, MA, 2006.

\bibitem{ch}
O. Christensen, 
{\it An introduction to frames and Riesz bases},
Applied and Numerical Harmonic Analysis. Birkh\"auser Boston, Inc., Boston, MA, 2003.


\bibitem{gm}
I. C. Gohberg, A. S. Markus, 
{\it Some relations between eigenvalues and matrix elements of linear operators},  Mat. Sb. (N.S.) {\bf 64} (1964), 481--496.

\bibitem{horn}
A. Horn,
{\it  Doubly stochastic matrices and the diagonal of a rotation matrix},
Amer. J. Math. {\bf 76} (1954), 620--630.

\bibitem{k1}
R. Kadison, 
{\it The Pythagorean theorem. I. The finite case},
Proc. Natl. Acad. Sci. USA {\bf 99} (2002), 4178--4184.

\bibitem{k2}
R. Kadison, 
{\it The Pythagorean theorem. II. The infinite discrete case},
Proc. Natl. Acad. Sci. USA {\bf 99} (2002), 5217--5222.

\bibitem{kw}
V. Kaftal, G. Weiss,
{\it An infinite dimensional Schur-Horn Theorem and majorization theory},  J. Functional Analysis (to appear).
 

\bibitem{kl}
K. Kornelson, D. Larson, 
{\it Rank-one decomposition of operators and construction of frames},
Wavelets, frames and operator theory, 203--214, Contemp. Math., {\bf 345}, Amer. Math. Soc., Providence, RI, 2004.

\bibitem{neu}
A. Neumann, 
{\it An infinite-dimensional version of the Schur-Horn convexity theorem},
J. Funct. Anal. {\bf 161} (1999), 418--451.

\bibitem{schur}
I. Schur,
{\it \"Uber eine Klasse von Mittelbildungen mit Anwendungen auf die Determinantentheorie}, 
Sitzungsber. Berl. Math. Ges. {\bf 22} (1923), 9--20.


\end{thebibliography}
\end{document}